\documentclass[11pt]{amsart}
\usepackage{pdfsync}

\usepackage{amsmath, amssymb}
\usepackage{tikz}
\usetikzlibrary{decorations.pathmorphing}

\newcounter{my_enumerate_counter}

\DeclareMathOperator{\cof}{cof}

\newcommand{\cP}{\mathcal P}

\newtheorem{theorem}{Theorem}
\newtheorem{thm}{Theorem}[section]
\newtheorem{prop}[thm]{Proposition}
\newtheorem{lemma}[thm]{Lemma}

\newtheorem{claim}[thm]{Claim}

\theoremstyle{definition} 
\newtheorem{example}[thm]{Example}

\newtheorem{definition}[thm]{Definition}

\newcommand{\bbC}{{\mathbb C}}

\newcommand{\bbN}{{\mathbb N}}

\newcommand{\cA}{{\mathcal A}}

\DeclareMathOperator{\Sp}{sp}

\newcommand{\cstar}{$\mathrm{C}^*$}
\newcommand{\cst}{\mathrm{C}^*}

\newcommand{\bfP}{\mathsf P}

\DeclareMathOperator{\supp}{supp}

\newcommand{\bfS}{\mathsf S}

\newcommand{\bfY}{\mathsf Y}

\newcommand{\fA}{\mathfrak A}

\newcommand{\rs}{\restriction}

\newcommand{\e}{\varepsilon}

\newcommand{\APS}{A_{+,1}} 

\newcommand{\sfZ}{\mathsf Z} 
\newcommand{\sfD}{\mathsf D}

\newcommand{\cPoH}{\cP_{\aleph_1}(H_\theta)}

\newcommand{\cPo}[1]{\cP_{\aleph_1}(#1)}
\newcommand{\cPotH}{\cP_{\aleph_2}(H_\theta)}
\newcommand{\MPC}{M^\perp\cap C}

\newcommand{\MPA}{(\bfP(A)\cap M)^\perp}

\newcommand{\McX}{\overline {M\cap X}}
\newcommand{\McXS}{\overline {M\cap X(S)}}
\newcommand{\NcX}{\overline {N\cap X}}

\DeclareMathOperator{\ASA}{A_{sa}}

\DeclareMathOperator{\COLL}{Coll}
\DeclareMathOperator{\COK}{\COLL(\aleph_1,<\kappa)}
\DeclareMathOperator{\COKK}{\COLL(\aleph_1,<j(\kappa))}

\newcommand{\sfK}{\mathsf K}

\title{Corson reflections}

\author{Ilijas Farah}

\address{Department of Mathematics and Statistics, York University, 4700 Keele Street, North York, Ontario, Canada, M3J 1P3} 
\email{ifarah@mathstat.yorku.ca}
\urladdr{http://www.math.yorku.ca/~ifarah}

\author{Menachem Magidor}

\address{The Hebrew University of Jerusalem\\
Einstein Institute of Mathematics\\
Edmond J. Safra Campus, Givat Ram\\
Jerusalem 91904, Israel}

\email{mensara@savion.huji.ac.il}

\date{\today}
\begin{document}

	\begin{abstract} 
		A reflection principle for Corson compacta holds in the forcing extension obtained by Levy-collapsing a supercompact cardinal to~$\aleph_2$. In this model, a compact Hausdorff space is Corson if and only if 
		all of its continuous images of weight~$\aleph_1$ are Corson compact. 
		We use the Gelfand--Naimark duality, and 
		our results are stated in terms of unital abelian \cstar-algebras. 
	\end{abstract}

\maketitle

Before starting, we should thank Alan Dow for pointing our attention to \cite{bandlow1991characterization}. 
Our use of \cstar-algebras is closely related to Bandlow's 
use of large Hilbert cubes. 
Similar methods have been used in  \cite{MR1270189}, \cite{MR1000971},  
\cite{eisworth2006elementary}, \cite{bandlow1994function}, \cite{dow1992set}, \cite{dow1995more}, 
\cite{kunen2003compact}, \cite{MR2352742},  and it is possible that 
the \cstar-algebraic vantage point may yield additional applications. 
A paper of Kunen 
(\cite{kunen2003compact}) contains a closely related analysis of 
Corson compact spaces.  Some of the ideas of this note are contained in their papers. 
Since we were not aware of these results, the present paper should be considered as a survey rather than a research article.

A compact Hausdorff space~$X$ is a \emph{Corson compactum} (or shortly,  Corson)
if it is homeomorphic to a subspace of some Tychonoff cube $[0,1]^\kappa$
which has the property that for every $\xi<\kappa$ 
the set $\{x\in X: x(\xi)\neq 0\}$ is countable. 

Every metrizable compactum is homeomorphic to a subspace of $[0,1]^\omega$ and therefore Corson. 
In \cite{magidor2017properties} it was proved that if there exists a non-reflecting stationary subset of cofinality $\omega$ ordinals in $\omega_2$, then 
there exists a compact Hausdorff space $X$ all of whose continuous images of weight $\aleph_1$ are Corson (and even uniform Eberlein; see \S\ref{S.Concluding}), but $X$ is not Corson. 

\begin{theorem} \label{T.1} 
Suppose $\kappa$ is a supercompact cardinal.  Then the following reflection statement holds in $V^{\COK}$: 
 If $X$ is a compact Hausdorff space, then  all continuous images of $X$ of 
 weight at most $\aleph_1$ are Corson compact if and only if $X$ is Corson. 
 The same principle follows from Martin's Maximum. 
  \end{theorem}

\subsection*{Acknowledgments} In addition to thanking Alan Dow again, we would like to thank G. Plebanek for helpful remarks on an early draft of this paper.

\section{\cstar-algebras and elementary submodels} 

\subsection{Background on \cstar-algebras} 
We quickly review the required results on \cstar-algebras, and  
unital, abelian  \cstar-algebras in particular. For additional information see e.g., 
\cite{Fa:STCstar},  \cite{Murphy:C*}
or  \cite{Black:Operator}. 

A  \emph{\cstar-algebra} is a complex Banach algebra with involution which is isomorphic to a norm-closed, self-adjoint, algebra of bounded linear operators on a complex Hilbert space. 
If $C$ is a \cstar-algebra and $Z\subseteq C$ then 
\[
\cst(Z)
\]
 denotes the \cstar-subalgebra of~$C$ generated by $Z$. 
 We write $\cst(a,Z)$ for $\cst(\{a\}\cup Z)$.

\subsubsection{Positivity} \label{S.Pos}
An element $a$ of a \cstar-algebra $A$ is \emph{self-adjoint} if $a=a^*$ and  \emph{positive} if it is self-adjoint and its spectrum is included in $[0,\infty)$. 
A standard argument (see \cite[II.3.1.3(ii)]{Black:Operator}) shows that 
$a$ is positive if and only if 
$a=b^*b$ for some $b\in A$. 
It is common to write $a\geq 0$ for `$a$ is positive'. 
For a \cstar-algebra~$A$ we write 
\begin{align*}
\ASA&=\{a\in A: a=a^*\},\\
A_+&=\{a\in A: a\geq 0\}, \\
\APS&=\{a\in A: 0\leq a\leq 1, \|a\|=1\}.
\end{align*}
On  $\ASA$ one defines  partial ordering by 
letting $a\leq b$ if and only if $b-a$ is positive. 

\subsubsection{States}\label{S.States} 
A continuous linear functional $\varphi$ on a \cstar-algebra 
 $C$ is \emph{positive} if $\varphi(a)\geq 0$ for every 
$a\in C_+$. A positive functional $\varphi$ is a \emph{state} if $\|\varphi\|=1$. 
If $C$ is a unital \cstar-algebra, with unit denoted  $1_C$, 
then a functional $\varphi$ is positive if and only if $\varphi(1)=\|\varphi\|$.  
The space of all states of a unital \cstar-algebra $C$ is convex and weak*-compact.

A proof of the following a well-known property of states is included for  reader's convenience (see \cite[Proposition~1.7.8]{Fa:STCstar} for a more general statement). 
 
\begin{lemma} \label{L.CS.0} Suppose that $\varphi$ is a state of a \cstar-algebra $A$ and $0\leq a$. 

If  $\varphi(a)=0$ then $\varphi(ab)=0$ for all $b$. 

If $\varphi(a)=\|a\|$ then $\varphi(ab)=\varphi(ba)=\|a\|\varphi(b)$ for all $b$. 
\end{lemma} 

\begin{proof} Fix $a\in \APS$ such that $\varphi(a)=1$. Then $0\leq 1-a\leq 1$ and therefore 
$(1-a)^2\leq 1-a$. 
Since $\varphi$ is positive, it satisfies $\varphi((1-a)^2)\leq \varphi(1-a)=0$. 
As it satisfies the Cauchy--Schwarz inequality, 
$|\varphi(d^*c)|^2\leq \varphi(c^*c)\varphi(d^*d)$ for all $c$ and $d$, 
we have 
\[
\varphi(b(1-a))\leq \varphi((1-a)^2)\varphi(b^*b)=0.
\]
By simplifying the left-hand side, this implies $\varphi(b)=\varphi(ba)$. The other equalities 
follow by a similar argument and by rescaling. 
\end{proof}

\subsubsection{Characters}\label{S.cps} 
A \emph{character} of a unital  \cstar-algebra is a unital *-ho\-mo\-mor\-phi\-sm $\varphi\colon C\to \bbC$. If $A$ is a unital \cstar-algebra, then its unit has an open neighbourhood consisting of invertible elements (\cite[Lemma~1.2.6]{Fa:STCstar}). Therefore,  every proper maximal two-sided ideal of $A$ is closed. In particular, the kernel of any character of $A$ is norm-closed.  All characters of a unital \cstar-algebras are therefore automatically continuous. 
Since every *-homomorphism is positive  and of norm at most 1, every character is  a state.

\subsubsection{The Gelfand--Naimark duality}\label{S.GN} \label{S.Spectrum}
 The category of compact Hausdorff spaces with respect to 
continuous maps as morphisms is equivalent to the category of unital, abelian \cstar-algebras with respect to 
$^*$-homomorphisms (i.e., homomorphisms which respect the adjoint operation) as morphisms. 
Given a compact Hausdorff space $X$, the \cstar-algebra associated with it is 
the $^*$-algebra of all complex continuous functions on $X$.  To a continuous map $f\colon X\to Y$
one associates the $^*$-homomorphism 
\[
C(Y)\ni a\mapsto a\circ f\in C(X). 
\] 
The inverse functor is defined as follows.
If $A$ is a unital abelian \cstar-algebra, then the space $X$ of characters 
of $A$ is compact in the weak*-topology and it separates points of $A$. 
The \emph{Gelfand transform} identifies $A$ with $C(X)$. 
By the Riesz Representation Theorem, the states on  $C(X)$ are in a bijective 
correspondence with the 
regular Radon probability 
 measures on $X$, via the correspondence $\varphi\mapsto \mu_\varphi$ where 
\[
\varphi(a)=\int_X a(x)\, d\mu_\varphi(x) . 
\]
  Every unital $*$-homomorphism 
$\Phi\colon C(X)\to C(Y)$ is of the  form $a\mapsto a\circ f$ for a continuous function $f\colon Y\to X$. 

The space of states of $C$ is denoted $\bfS(C)$. 
The \emph{pure states} are the extreme points of $\bfS(C)$. A state $\varphi$ on $C(X)$ is pure if and only if the 
associated probability measure $\mu_\varphi$ is a \emph{point-mass measure} (i.e., a measure that concentrates on a single point). In this case $\varphi$ agrees with 
the evaluation functional $a\mapsto a(x)$ for $x\in X$.  
The space of pure states of $C$, also known as the \emph{spectrum} of $C$,  is denoted $\bfP(C)$.

When $C$ is abelian, then a state is pure if and only if it is a character.  

We will also need the following well-known lemma (the weight of a topological space is the smallest cardinality of its basis, and 
the density character of a metric space is the minimal cardinality of a dense subspace).  

\begin{lemma}\label{L.weight} The weight of an infinite compact Hausdorff space $X$ is equal to the density character of $C(X)$. 
\end{lemma}

\begin{proof} 
The weight of an infinite compact Hausdorff space $X$ is equal to the minimal cardinal $\kappa$ such that $X$ is homeomorphic to a subspace 
of $[0,1]^\kappa$. The density character of $C(X)$ is equal to the cardinality of a minimal generating set for $C(X)$, which 
is by the Stone--Weierstrass theorem equal to the minimal cardinality of a subset of $C(X)$ that separates points of~$X$. 
But this is the minimal $\kappa$ such that $X$ is homeomorphic to a subspace of~$[0,1]^\kappa$. 
\end{proof} 

\subsubsection{Continuous Functional Calculus}\label{S.cfc} 
The \emph{spectrum} of an  element $a$ of a unital \cstar-algebra $C$,
 is\footnote{Since $C$ is unital, we identify the 
scalar multiples of its unit $1_C$ with the complex numbers.} 
\[
\Sp(a)=\{\lambda\in \bbC: \lambda-a \text{ is not invertible}\}. 
\]
Two nontrivial facts deserve mention. 
The spectrum of any operator is a nonempty compact subset of the field of complex numbers. 
Second,  if $A\subseteq B$ are unital \cstar-algebras with the same unit and $a\in A$, 
then $\Sp(a)$ as computed in $A$ is equal to $\Sp(a)$ as computed in $B$. 
A normal element $a$ is positive if and only if 
 its spectrum is included in $[0,\infty)$.

If $a\in C(X)$ then clearly $\Sp(a)$ is equal to the range of $a$. 
An element $a$ of a \cstar-algebra is  \emph{normal} 
if $aa^*=a^*a$. 
The \emph{continuous functional calculus} 
asserts that $C(\Sp(a))\cong \cst(a,1)$, 
via the isomorphism defined by $f\mapsto f(a)$.

\subsection{Corson compacta}

We will need the following standard properties of Corson compacta. 
\begin{enumerate}
	\item Every closed subspace of a Corson compactum is a Corson compactum. 
	\item Every Corson compactum $X$ has the following property. If $Z\subseteq X$ and $x$ is an accumulation point of $Z$, 
	then there exists a sequence~$z_n\in Z$, for $n\in \bbN$, such that $\lim_n z_n=x$. 
	A space with this property is said to be  \emph{Fr\'echet} or \emph{Fr\'echet--Urysohn}.  
	\item Every continuous image of a Corson compactum is a Corson compactum.  
\end{enumerate}
The third property is \cite[Theorem~6.2]{michael1977note}, and the first two are straightforward.

\subsection{Corson compacta and \cstar-algebras}
\label{S.Corson} 
The following, almost tautological, lemma provides  reformulation of our results in terms of \cstar-algebras.  

\begin{lemma} \label{L.Corson} 
If $X$ is a compact Hausdorff space then $X$ is Corson compact
if and only if there are a cardinal $\kappa$ and a family $a_\alpha$, $\alpha<\kappa$, in $C=C(X)$ 
such that the following conditions hold (see \S\ref{S.Pos} for the notation $\leq$). 
\begin{enumerate}
\item   \label{1.L.Corson}For every $\alpha$ we have 
$0\leq a_\alpha\leq 1$. 
\item  \label{2.L.Corson} For every pure state $\varphi$ on $C$ the set
$
\{\alpha<\kappa: \varphi(a_\alpha)\neq 0\}
$
is countable. 
\item  \label{3.L.Corson}The \cstar-algebra  $C$ is generated by $\{a_\alpha: \alpha<\kappa\}$ and $1$. 
\end{enumerate}
\end{lemma} 

\begin{proof} Suppose  $X$ is a Corson compactum. Therefore we may identify $X$ with a subspace of   
a Hilbert cube $ [0,1]^\kappa$ such that $\{\alpha<\kappa: x(\alpha)\neq 0\}$ is countable for all $x\in X$. 
Then for $\alpha<\kappa$ 
the projection $a_\alpha$ to the $\alpha$th coordinate is a continuous function
from $X$ into $[0,1]$. These functions   
separate the points of $X$, and therefore the complex Stone--Weierstrass 
theorem implies  
 $C(X)\cong \cst(\{1\}\cup \{a_\lambda: \lambda<\kappa\})$. 
The condition \eqref{2.L.Corson}  is clearly equivalent to the assertion that for every 
$\lambda<\kappa$ the set $\{x\in X: x(\lambda)\neq 0\}$ is countable. Since the pure states of $C(X)$ are exactly the evaluation functions at the points of $X$ (see \S\ref{S.Spectrum}), 
this proves that every pure state $\varphi$ vanishes at all but countably many of $a_\lambda$.

Conversely, if 
some family $a_\alpha$, for $\alpha<\kappa$, 
in $C(X)$ 
satisfies \eqref{1.L.Corson}--\eqref{3.L.Corson}, then the function 
from $X=\cP(C(X))$ defined by 
\[
\varphi\mapsto \langle \varphi(a_\alpha): \alpha<\kappa\rangle
\]
 is a homeomorphism onto a Corson compact subspace of $[0,1]^\kappa$. 
\end{proof} 

An proof analogous to that of Lemma~\ref{L.Corson} gives the following. 

\begin{lemma} \label{L.Eberlein} 
	If $X$ is a compact Hausdorff space then $X$ is Eberlein compact
	if and only if there are a cardinal $\kappa$ and a family $a_\alpha$, $\alpha<\kappa$, in $C=C(X)$ 
	such that the following hold (see \S\ref{S.Pos} for the notation $\leq$). 
	\begin{enumerate}
		\item   \label{1.L.E}For every $\alpha$ we have 
		$0\leq a_\alpha\leq 1$. 
		\item  \label{2.L.E} 
		The set
		$
		\{\alpha<\kappa: \varphi(a_\alpha)>\e\}
		$
		is finite for every pure state $\varphi$ on $C$  and every $\e>0$. 
		\item  \label{3.L.E}The \cstar-algebra  $C$ is generated by $\{a_\alpha: \alpha<\kappa\}$ and $1$. \qed
	\end{enumerate}
\end{lemma} 

This is a good moment to introduce the following lemma, needed in the proof of  Theorem~\ref{T.1}.

\begin{lemma} \label{L.aleph1}
If $X$ is a Corson compactum of weight at most $\lambda$ then the cardinality of $X$ is not greater than $\lambda^{\aleph_0}$. In particular, if the Continuum Hypothesis holds and $\lambda\leq\aleph_1$, then $|X|\leq \aleph_1$. 
\end{lemma} 

\begin{proof} Let $\kappa$ be a cardinal such that $X$ is homeomorphic to 
a subspace of $[0,1]^\kappa$ such that 
\[
\supp(x)=\{\alpha: x(\alpha)\neq 0\}
\]
 is countable for all $x\in X$. We claim that there exists 
$S\subseteq \kappa$ with  $|S|\leq\lambda$ such that the evaluation functions at the points of $S$ 
separate the points of $X$. To prove this, fix a pair of basic open sets $U$ and $V$ of $X$ such that $\overline V\subseteq U$. By compactness, there is a finite list of Tychonoff basic open subsets of $[0,1]^\kappa$, $W_i=W_i(U,V)$, for $i<m(U,V)$, such that $\overline V\subseteq \bigcup_i W_i$ and $\bigcup_i W_i\cap X\subseteq U$. There are $\lambda$ such pairs $U,V$ and each $W_i$ depends on a finite set of coordinates in~$\kappa$. Let $S$ be the set of all these coordinates. Then its cardinality is not greater than  $\lambda$ and it separates the points of $X$. The projection to $[0,1]^S$ is a continuous injection on $X$. By compactness (and Hausdorffness) of $X$, it is a homeomorphism. 
Therefore, $X$ is homeomorphic to a subset of $[0,1]^\lambda$ as in the definition of Corson compacta.  
For a countable $A\subseteq \lambda$ let $X_A=\{x\in X: \supp(x)\subseteq A\}$ (with $\supp(x)$ re-evaluated as a subset of~$\lambda$).  
Then $|X_A|\leq |[0,1]|^{|A|}=2^{\aleph_0}$. 
Since $X$ is Corson, we have $X=\bigcup_A X_A$ and therefore $|X|\leq \lambda^{\aleph_0}\cdot  2^{\aleph_0}=\lambda^{\aleph_0}$, as required. 
The second claim follows immediately. 
\end{proof}

\section{Elementary submodels}

Throughout this section $C$ is a unital \cstar-algebra. It is assumed to be abelian and equal to $C(X)$ for some 
compact Hausdorff space $X$  unless 
otherwise specified. By \S\ref{S.Spectrum}, we can identify  $X$ with the space $\bfP(C)$ of all pure states of $C$. 
 A large enough regular cardinal $\theta$ is fixed. 
The set~$H_\theta$ of all sets of hereditary cardinality strictly less than $\theta$ is a model of a  fragment 
of ZFC sufficiently large for many practical applications (see e.g.,~\cite{Ku:Set}). 
Although elementary submodels are have been an important tool in general topology 
for years (see \cite{dow1988introduction}), this is to the best of our knowledge the first time that they are applied to study of 
compact Hausdorff spaces via the Gelfand--Naimark duality. 
 
\begin{definition} \label{Def.MC} 
Suppose that $C$ is a \cstar-algebra (unital and not necessarily abelian) 
and $M$ is an elementary submodel (not necessarily countable) of a large enough $H_\theta$ such that $C\in M$. We write 
\begin{align*}
C_M&=\cst(M\cap C),\\
\bfP(C)^M& = \bfP(C)\cap M.
\end{align*}   
For $a\in C$ and $\bfY\subseteq \bfS(C)$ write 
\[
\|a\|_{\bfY}=\sup_{\varphi\in \bfY}|\varphi(a)|
\]
and let 
$\|a\|_M=\|a\|_{M\cap \bfP(C)}$. 

By a minor abuse of notation, we denote the   \emph{annihilator} of $\bfP(C)\cap M$ in~$C$ by    $\MPC$, so that 
\[
\MPC =\{a\in C: \|a\|_M=0\}. 
\]
\end{definition}

The following is a consequence of the definitions. 

\begin{lemma} If $M$ and $C=C(X)$ are as in  Definition~\ref{Def.MC} 
then we have $\MPC=\{a\in C: a(x)=0$ for all $x\in X\cap M\}$. \qed
\end{lemma} 

\begin{lemma} If $M$, $C=C(X)$, and $Y\subseteq\bfS(C(X))$,   are as in  Definition~\ref{Def.MC} then 
 $\|\cdot\|_{\bfY}$ (and $\|\cdot\|_M$ in particular) is a seminorm majorized by 
$\|\cdot \|$ on $C$. 
\end{lemma} 
\begin{proof}  
For every state~$\varphi$ of $C$ we have $\|\varphi\|=1$ and therefore     
 $|\varphi(\cdot)|$ is a seminorm majorized by $\|\cdot \|$. Therefore 
 $\|\cdot\|_{\bfY}$ is the supremum of a family of seminorms majorized   by $\|\cdot\|$. 
 \end{proof}

An \emph{order ideal} in a \cstar-algebra $A$ is a subset $\cA$ of $A_+$ that is a \emph{cone} (i.e., closed under the multiplication 
by positive scalars and addition) and \emph{hereditary} (i.e., if $a\in \cA$ and $0\leq b\leq a$, then $b\in \cA$).

\begin{lemma} \label{L.perp} 
	Suppose that $A$ is a (not necessarily commutative) 
	\cstar-algebra and  $M$ is an elementary submodel of a large enough $H_\theta$ such that $A\in M$. 
	Then 
	the annihilator of $\bfP(A)^M$, 
	\[
(\bfP(A)\cap M)^\perp	=\{a\in A: \|a\|_M=0\}
	\]
	is a norm-closed subspace of $A$ and its positive cone, $(\MPA )_+$, 
	is a norm-closed  order ideal in $A$.
	
	If $A$ is in addition abelian, then $\MPA$ is an ideal of $A$, and the quotient $C/(\MPA)$ is isomorphic to $C(\McX)$, 
	where $\McX$ is considered with the subspace topology. 
\end{lemma}  

\begin{proof} For the first part, 
	the annihilator of any subset of the dual space of a Banach space $Z$ is a norm-closed subspace of $Z$.  
	
	To prove the second part, note that the zero set 
	$\{a\in A_+: \varphi(a)=0\}$ of any state $\varphi$ 
	 is a norm-closed cone. 
	Hence $\MPA $ is an intersection of a family of cones, and therefore a norm-closed cone itself. 
	
	A \cstar-subalgebra of a \cstar-algebra is a left ideal if and only if its positive part is an order ideal (this is a result of 
	Effros, see \cite[Theorem~1.5.2]{pedersen2018c}).

	Now suppose $A$ is abelian.  Then every left ideal  of $A$ is an ideal of $A$. 
	It remains to prove $C/(\MPA)\cong C(\McX)$. Consider the *-homomorphism $\pi_M\colon C\to C(\McX)$ defined by 
	\[
\pi_M(a)=a\rs \McX.
\] 
	This is a surjection of $C$ onto $C(\McX)$, and its kernel is equal to $\{a\in C: a\rs\McX=0\}=\MPA$. 
\end{proof}

By Lemma~\ref{L.perp}, if $C$ is a unital  abelian \cstar-algebra and $M\prec H_\theta$ has $C$ as an element, then 
we have an exact sequence
%
 
 \begin{tikzpicture}
 \matrix[row sep=.3cm,column sep=.7cm]{
 	\node (0) {$0$}; 
&  	\node (Mperp) {$\MPC$};
 	& \node (C) {$C$}; 
&\node (CMX){$C(\McX)$}; 
& \node (00) {$0$}; \\
 };
\draw (0) edge [->] (Mperp); 
 \draw (Mperp) edge [->] (C); 
 \draw (C) edge [->] node [above] {$\pi_M$} (CMX); 
 \draw (CMX) edge [->] (00); 
 \end{tikzpicture}
 
What is the relation of $C_M$ to the algebras in this exact sequence? 
With $\iota_M\colon C_M\to C$  denoting the inclusion map, we have the following commutative diagram. 

\begin{tikzpicture}
\matrix[row sep=.3cm,column sep=.7cm]{
& & \node (CMX){$C(\McX)$}; \\
\node (Mperp) {$\MPC$};
& \node (C) {$C$}; \\
& & \node (CM) {$C_M$}; \\
};
\draw (Mperp) edge [->] (C); 
\draw (C) edge [->] node [above] {$\pi_M$} (CMX); 
\draw (CM) edge [->] node [below] {$\iota_M$} (C);
\draw (CM) edge [->] node [right] {$\pi_M\circ \iota_M$}(CMX); 
\end{tikzpicture}

A model $M$ that satisfies any of the equivalent conditions in Lemma~\ref{L.bad}  is said to \emph{split} $X$. 
In \S\ref{S.Splits} we will discuss this notion in some depth (not needed in the proof of Theorem~\ref{T.1}). 

\begin{lemma}\label{L.bad} 
	Suppose $X\in H_\theta$ is a compact Hausdorff space, $M\prec H_\theta$, and $X\in M$. 
	Then the following are equivalent 
	\begin{enumerate}
		\item\label{1.L.bad}  $\cst(C_M, \MPC)=C$. 
		\item \label{2.L.bad} If $x$ and $y$ are in $\McX$ and satisfy $a(x)=a(y)$ for all $a\in M\cap C$, then $x=y$. 
		\item \label{4.L.bad} $\pi_M\circ \iota_M$ is a surjection of $C_M$ onto $C(\McX)$, and the exact 
		sequence $0\to \MPC\to C(X)\to C(\McX)\to 0$ splits. 
	\end{enumerate}
\end{lemma}

\begin{proof} 
	\eqref{4.L.bad} $\rightarrow$ \eqref{1.L.bad}: Suppose $\pi_M\circ \iota_M$ is a surjection. Let us first show that this 
	implies that the exact sequence in \eqref{4.L.bad} splits. 
	Since $C_M$  separates points of $\McX$, $\pi_M\circ \iota_M$ is an injection and therefore an isomorphism between $C_M$ and $C(\McX)$. 
With $\theta\colon C(\McX)\to C_M$ denoting the inverse of $\pi_M\circ \iota_M$, we have a split exact sequence 

\begin{tikzpicture}
\matrix[row sep=.3cm,column sep=.7cm]{
	\node (0) {$0$}; 
	&  	\node (Mperp) {$\MPC$};
	& \node (C) {$C$}; 
	&
	& \node (CMX) {\qquad$C(\McX)$}; 
	& \node (00) {$0$}; \\
};
\draw (0) edge [->] (Mperp); 
\draw (Mperp) edge [->] (C); 
\draw (C) edge [bend left, ->] node [above] {$\pi_M$} (CMX); 
 \draw (CMX) edge [bend left, ->] node [below] {$\iota_M\circ\theta$} (C);
\draw (CMX) edge [->] (00); 
\end{tikzpicture}

In order to prove \eqref{1.L.bad}, fix $a\in C$ and let $a_1=(\iota_M\circ \theta)(a)$.  Then $a_1\in C_M$, 
$a_0=a-a_1$ belongs 
to $\MPC$, and therefore $a=a_0+a_1$ belongs to $C(C_M,\MPC)$
	
	\eqref{1.L.bad} $\rightarrow$ \eqref{2.L.bad}: Suppose that \eqref{2.L.bad} fails and fix distinct $x$ and $y$ in $\McX$ 
	such that $a(x)=a(y)$ for all $a\in M\cap C$. 
	Since $x\neq y$, there exists $b\in C$ such that $b(x)\neq b(y)$. But every $c\in (\McX)^\perp$ satisfies $c(x)=0=c(y)$, 
	and therefore $b\notin \cst(M\cap C,(\McX)^\perp)$, showing that \eqref{1.L.bad} fails.

\eqref{2.L.bad} $\rightarrow$ \eqref{4.L.bad} The assumption asserts that the elements of $C_M$ separate points of $\McX$. 
Therefore $(\pi_M\circ \iota_M)(C_M)$ is a norm-closed, self-adjoint, subalgebra of $C(\McX)$ that separates points and contains all constant functions. 
By the complex Stone--Weierstrass theorem (e.g., \cite[Theorem~4.3.4]{Pede:Analysis}), it is equal to $C(\McX)$. 
\end{proof}

The following proposition ought to be well-known. 

\begin{prop} \label{P.Angelic} Suppose that $X$ is a compact Hausdorff space such that every continuous image of $X$ of weight at most $2^{\aleph_0}$ is
	Fr\'echet. Then $X$ is Fr\'echet. 
\end{prop} 

\begin{proof} Fix $Z\subseteq X$ and $x\in \overline Z$. In order to find a sequence in $Z$ that converges to $x$, 
	fix a large enough regular cardinal $\theta$ and $M\prec H_\theta$ that contains $X,Z$, and $x$ and such that $M^\omega\subseteq M$ and $|M|=2^{\aleph_0}$. Then the pure state space of~$C_M$ is, being a continuous image of $X$ of weight at most $2^{\aleph_0}$, Fr\'echet. 
	We can identify all $z\in Z\cap M$ and $x$ with pure states of $C_M$. 
	
	\begin{claim} In the weak*-topology induced by $C_M$, $x$ is an accumulation point of $Z\cap M$. 
	\end{claim} 

\begin{proof} 
	In the weak*-topology induced by $C$, $x$ is an accumulation point of~$Z$. 	
	This means that for all $n\geq 1$ and all $a_j\in C$, for $j<n$, the $n$-tuple 
	$(a_j(x): j<n)$ is an accumulation point of $\{(a_j(z): j<n): z\in Z\}$. 
	Since $M\prec H_\theta$, the following holds for all $n$:   
	
	For all $a_j\in C(X)\cap M$, for $j<n$,  
	the $n$-tuple $(a_j(x): j<n)$ is an accumulation point of $\{(a_j(z): j<n): z\in Z\cap M\}$. 

	Since $C(X)\cap M$ is dense in $C_M$, $x$ is an accumulation point of $Z\cap M$ in the weak*-topology induced by $C_M$. 
\end{proof} 

	Let $z_n$, for $n\in \bbN$, be  a sequence in $Z\cap M$ that converges to $x$ in the weak*-topology of $C_M$. 
	This sequence belongs to $M$, since $M^\omega\subseteq M$. 
	By elementarity, $\lim_n a(z_n)= a(x)$ for all $a\in C$.  This implies $\lim_n z_n=x$ in the topology of $X$. 

Since $Z$ and $x$ were arbitrary, this proves that $X$ is Fr\'echet. 
\end{proof} 

The following result appears as \cite[Exercise~2.4G]{engelking1989general} and we include a proof for reader's convenience. 

\begin{lemma} \label{L.Angelic} A continuous image of a compact, Hausdorff, and Fr\'echet space is Fr\'echet. 
	\end{lemma} 

\begin{proof} Suppose $X$ is Fr\'echet and $f\colon X\to Y$ is a surjection. Fix $Z\subseteq Y$ and an accumulation point $x$ of $Z$. 
	Let $Z'=f^{-1}(\{x\})$ and let $T=f^{-1}(Z)$.  We claim that $\overline {Z'}\cap T\neq \emptyset$. 
	Assume otherwise, and for every $z\in T$ fix an open neighbourhood $u_z$ disjoint from $Z'$. 
Then $U=\bigcup_{z\in T} u_z$  is an open cover of~$T$ disjoint from $Z'$, and $f[X\setminus U]$ is a compact subset of $Y$ containing $Z$ that $x$ does not belong to; 
contradiction. 

Therefore $\overline{Z'}\cap T\neq \emptyset$. Since $X$ is Fr\'echet, there exists a sequence $(z_n')$ in $Z'$ such that $x'=\lim_n z_n'$ belongs to $T$. 
Then $f(z_n')\in Z$ for all $n$ and $\lim_n f(z_n')=x$. 

Since $Z$ and $x$ were arbitrary, this proves that $Y$ is Fr\'echet. 
\end{proof} 
	
\section{Proof of  Theorem~\ref{T.1}} \label{S.proof.T.1}

Some $N\prec H_\theta$ of cardinality $\aleph_1$ is said to be \emph{internally approachable} if 
it is equal to the union of an increasing $\omega_1$-chain of countable 
elementary submodels each of whose proper initial segments belongs to the next model in the sequence.    
Consider the following reflection principle. 

\begin{enumerate}
	\item [(R)] If $\theta$ is an uncountable regular cardinal and 
	$S\subseteq \cPoH$ is stationary, then there exists an internally approachable  $N\prec H_\theta$ of cardinality~$\aleph_1$
	such that $S\cap \cPo N$ is stationary in $\cPo N$. 
\end{enumerate}
This principle was introduced and proved to follow from MM in \cite[Theorem~13]{FoMaShe:Martin}. 
Requiring $N$  to be an elementary submodel of $H_\theta$ is not a loss of generality, 
since if every stationary subset of $\cPoH$ reflects to $\cPo Z$ for some $Z\in \cPotH$, then 
every stationary subset of $\cPoH$ reflects to a stationary set of $Z\in \cPotH$.

Proposition~\ref{P.Ref} is standard, but we could not find it in the literature. A proof is included for reader's convenience.

\begin{prop} \label{P.Ref} If $\kappa$ is a supercompact cardinal then $\COK$ forces the following strengthening of (R).
\begin{enumerate}
	\item [] If $\theta$ is an uncountable regular cardinal and 
	$S\subseteq \cPoH$ is stationary, then there exists $N\prec H_\theta$ of cardinality $\aleph_1$ closed under $\omega$-sequences	such that $S\cap \cPo N$ is stationary in $\cPo N$. 
\end{enumerate}
\end{prop} 
	
\begin{proof} Every model of cardinality $\aleph_1$ closed under $\omega$-sequences is clearly approachable. We will prove a strengthening in which the model $N$ is required to be  closed under $\omega$-sequences. 
Let $V[G]$ be a forcing extension by $\COK$. Suppose 
	$\theta$ is an uncountable regular cardinal and 
	$S\subseteq \cPoH$ is stationary. Let $j\colon V\to N$ be an elementary embedding with critical point $\kappa$ such that  $j(\kappa)>\theta$ and~$N$ is closed under $2^{<\theta}$-sequences. (We will be using the fact that $|H_\theta|=2^{<\theta}$.) 
	
	 By the standard methods (\cite[Proposition~9.1]{cummings2010iterated}), $j$ can be extended to an elementary embedding (also denoted $j$) 
	 \[
	 j\colon V[G]\to N[G_1]
	 \]
	 for an $N$-generic filter $G_1\subseteq \COKK$ such that 
	 \[
	 G_1\cap \COK=G. 
	 \]
	Let $Z=(H_\theta)^{V[G]}$.
	 Since $\theta\geq \kappa$ is regular and $\COK$ has $\kappa$-cc, $H_\theta^V$ is a $\COK$-name for $Z$. 
		As $N$ is closed under $2^{<\theta}$-sequences and $\COK$ has $\kappa$-cc, $Y=j``Z$ belongs to $N[G_1]$. 
		It is an elementary submodel of $j(Z)=H_{j(\theta)}$. Since $\COKK$ is $\aleph_1$-closed, $\cPo{Y}\subseteq Y$. 
	In $V[G]$ it holds that $S$ is a stationary subset of $\cPo Z$, and the quotient forcing $\COLL(\aleph_1,[\kappa,j(\kappa)))$ is $\aleph_1$-closed. 
	Therefore  $S$  remains a stationary subset of $\cPo Z$ in $N[G_1]$. 
	But this is equivalent to $j``S$ being stationary in $\cPo Y$.  

In $N[G_1]$ we therefore have  $|Y|=\aleph_1$, $Y^\omega\subseteq Y$, and $j(S)$ reflects to $Y$. 
By elementarity, in $V$ there exists $X\prec H_\theta$ of cardinality $\aleph_1$
	such that $S\cap \cPo Y$ is stationary in $\cPo X$ and $X^\omega\subseteq X$. 
	
	Since $\theta$ and $S$ were arbitrary, we have proved that (R) holds in $V[G]$. 
	\end{proof}	
	
\begin{lemma} \label{L.split} 
	Suppose that (R) holds and $X$ is a compact Hausdorff space 
	such that every continuous image of $X$ of weight not greater than $\aleph_1$ 
	is Corson. If $\theta$ is a regular cardinal such that $X\in H_\theta$ then the set
	\[
	\sfD=\{M\in \cPoH:  X\in M, M\text{ does not split } X\}
	\]
	is nonstationary. 
\end{lemma} 

\begin{proof} Assume otherwise. By (R)  fix an internally approachable 
$N\prec H_\theta$ of cardinality $\aleph_1$ such that  $\sfD\cap \cPo N$ is stationary in $\cPo N$. 
	Then $\overline{N\cap X}$ is a closed subspace of $X$ of weight not greater than $\aleph_1$, and therefore a Corson compactum. 
	
	By Lemma~\ref{L.Corson}  we can 
	fix $a_\alpha$, for $\alpha<\aleph_1$, that together with $1$ generate~$C_N$ so that the set 
	\[
	Z(x)=\{\alpha<\aleph_1: a_\alpha(x)\neq 0\}
	\]
	is countable for all $x\in \NcX$. 

At this point we cannot assert that this sequence of generators belongs to~$N$ (we cannot even assert that any of the generators belongs to $N$). 
	Since $N$ is internally approachable, we can choose a continuous $\aleph_1$-sequence $M_\alpha$, for $\alpha<\aleph_1$, of countable elementary submodels of the expanded structure $(N,(a_\alpha: \alpha<\aleph_1))$ such that 
	$N=\bigcup_\alpha M_\alpha$. We say that an ordinal $\alpha<\aleph_1$ is \emph{good} if the following two conditions hold. 
	\begin{enumerate}
	\item \label{1.good} If $\beta<\alpha$ then $a_\beta\in C_{M_\alpha}$ 
	\item  \label{2.good} If $x\in M_\alpha\cap X$, then $Z(x)\subseteq \alpha$. 
	\end{enumerate}
A standard closing off argument shows that  the set of good $\alpha$ includes a club in $\aleph_1$. 
	Since $\sfD\cap N$ is stationary, there exists a good $\alpha$ such that $M_\alpha\in \sfD$. 
	As $M_\alpha$ does not split $X$, some distinct $x$ and $y$ in $\overline{M_\alpha\cap X}$ satisfy $a(x)=a(y)$ for all $a\in M_\alpha\cap C$. Since $M_\alpha\in N$, by elementarity we can choose $x$ and $y$ to be elements of $N\cap X$. Since $a(x)=a(y)$ for all $a\in C_{M_\alpha}$,  \eqref{1.good} implies that $a_\beta(x)=a_\beta(y)$ for all $\beta\in M_\alpha\cap \aleph_1$.

Since $\overline{N\cap X}$ is Corson, it is  Fr\'echet, and there are $x_n$ and $y_n$, for $n<\omega$, in $M_\alpha\cap X$ such that, in the weak*-topology of $C_N$, 
we have  $\lim_n x_n=x$ and $\lim_n y_n=y$. 
	By \eqref{2.good},  $Z(x_n)\subseteq M_\alpha$ and $Z(y_n)\subseteq M_\alpha$ for all $n$. Therefore 
	for   $\gamma\in \aleph_1\setminus M_\alpha$ we have $a_\gamma(x_n)=a_\gamma(y_n)=0$ for all $n$,  hence $a_\gamma(x)=a_\gamma(y)=0$.

	We have proved that  $a_\gamma(x)=a_\gamma(y)$ for all $\gamma<\aleph_1$. However, $x$ and $y$ belong to $N$,  $C_N$ separates points of $N\cap X$, and $C_N$ is generated by $a_\gamma$, for $\gamma<\aleph_1$; contradiction. 
\end{proof}

\begin{lemma} \label{L.K} 
	Suppose that $\theta$ is a regular cardinal, $X\in H_\theta$ is a compact, Fr\'echet, Hausdorff space, and the set 
	\[
	\sfK=\{M\in \cPoH: \text{$X\in M$ and  $M$ splits $X$}\}
	\]
includes a club. If $f\colon H_\theta^{<\omega}\to H_\theta$ is such that every $M\in \cPoH$ closed under $f$ is in $\sfK$, 
then every $N\prec H_\theta$  closed under $f$ and such that $X\in N$ splits $X$. 
\end{lemma} 

\begin{proof} 
	Fix $N\prec H_\theta$ such that $X\in N$ and $N$ is closed under $f$. 
	Suppose that $x$ and $y$ are distinct points of $\NcX$. 
	Since $X$ is Fr\'echet, there are sequences $(x_n)$ and $(y_n)$ in $N\cap X$ converging to $x$ and $y$, 
	respectively. Fix a countable $M\prec N$ closed under $f$ and such that $X$ and all $x_n$ and all $y_n$ belong to $M$. 
	Then $x$ and $y$ belong to $\McX$. Since $M$ splits $X$, there is $a\in C_M$ such that $a(x)\neq a(y)$. 
	Since $M\subseteq N$, we have $a\in N$. 
	
	Because $x$ and $y$ were arbitrary distinct points of $\NcX$, we conclude that $N$ splits $X$. 
\end{proof}

\begin{lemma} \label{L.Recursive} 
Suppose $X\in H_\theta$ is a compact, Fr\'echet,  Hausdorff space and~$M_\alpha$, for $\alpha<\lambda$, 
is a continuous chain of elementary submodels with the following properties for all $\alpha<\lambda$:
\begin{enumerate}
	\item $X\in M_0$. 
	\item $M_{\alpha}$ splits $X$. 
	\item $\overline{M_{\alpha}\cap X}$ is a Corson compact subspace of $X$. 
	\item $\bigcup_{\alpha<\lambda} M_\alpha\supseteq X$. 
\end{enumerate}
Then $X$ is Corson compact. 
\end{lemma} 

\begin{proof} We write $C=C(X)$, and we also write $C_\alpha$ in place of $C_{M_\alpha}$. 
	Let us say that a sequence $(a_\alpha)_{\alpha<\kappa}$ in a unital, abelian, \cstar-algebra $D$ 
	that satisfies \eqref{1.L.Corson}--\eqref{3.L.Corson} of Lemma~\ref{L.Corson}
	is a sequence of \emph{Corson generators} for $D$. We will choose a sequence of Corson generators for $C$ in blocks, by recursion on $\lambda$. 

	First choose a sequence $\cA_0$ of Corson generators for $C_0$. 
	For $\alpha<\lambda$, let 
	\[
	D_{\alpha+1}=\cst(1,M_\alpha^\perp\cap C_{\alpha+1}).
	\] 
	This is a unital \cstar-subalgebra of $C_{\alpha+1}$. Since the pure state space of $C_{\alpha+1}$ is Corson 
	and every continuous image of a Corson compact space is Corson, the pure state space of $D_{\alpha+1}$ is Corson. 
	We can therefore choose a sequence~$\cA_{\alpha+1}$ of Corson generators for $D_{\alpha+1}$. For a limit ordinal $\alpha$ let~$\cA_\alpha=\emptyset$. 

We claim that $\cA=\bigcup_\alpha \cA_\alpha$ is a sequence of Corson generators for $C(X)$. 	

First we prove that $\cst(\cA\cup \{1\})=C$. Towards this end, we use induction on $\alpha\leq \kappa$
to prove that $\cst(\bigcup_{\beta\leq\alpha} \cA_\beta\cup \{1\})=C_\alpha$.

This is true for $\alpha=0$. 
Suppose that the assertion is true for all $\beta<\alpha$. 
Consider the case when $\alpha$ is a successor ordinal, say $\alpha=\beta+1$. 
Since $M_\beta$ splits $X$, 
we have 
\[
\textstyle C_{\beta+1}=\cst(C_\beta, M_\beta^\perp)=\cst(\cA_{\beta+1}\cup \bigcup_{\gamma\leq \beta}\cA_\gamma)
\]
as required. 
If $\alpha$ is a limit ordinal, then  $\bigcup_{\beta<\alpha} C_\beta$ 
is dense in $C_\alpha$, and $\cA_\alpha=\bigcup_{\beta<\alpha}C_\beta$  generates~$C_\alpha$. 

Therefore $\cst(\cA)=C$, and it remains to prove that the set 
\[
Z(x)=\{a\in \cA: a(x)\neq 0\}
\]
is countable for every $x\in X$. 
Assume otherwise and fix $x$ such that $Z(x)$ is uncountable.  
Since $\cA_\alpha$ is a sequence of Corson generators for every $\alpha$, 
$Z(x)\cap \cA_\alpha$ is countable for all $\alpha$. 
Therefore the set $\{\alpha: Z(x)\cap \cA_\alpha\neq \emptyset\}$ is uncountable. 
Let $\beta$ be the least limit point of this set of cofinality $\aleph_1$. 
Let $\varphi$ be the pure state of~$C_\beta$ corresponding to $x$, i.e., $\varphi(a)=a(x)$ for $a\in C_\beta$. 
Lemma~\ref{L.Angelic} implies that $\cP(C_\beta)$ is Fr\'echet. Since $\cof(\beta)$ is uncountable, there exists $\gamma<\beta$
such that~$\varphi$ belongs to $\overline{M_\gamma\cap X}$. 
By construction,  this implies that $\cA_\delta$ is annihilated by $x$ for all $\alpha\leq \delta<\beta$; contradiction. 
\end{proof} 


A discussion on why Lemma~\ref{L.Recursive} does not apply to the space constructed in \cite{magidor2017properties} is given in 
Example~\ref{Ex.MP}.

\begin{proof}[Proof of Theorem~\ref{T.1}] By Proposition~\ref{P.Ref}, (R)
	holds in $V^{\COK}$. This model also satisfies the Continuum Hypothesis. 
By induction on~$\lambda$ we will prove that
if $X$ is a compact Hausdorff space of weight $\lambda$ such that all continuous images of $X$ of weight at most $\aleph_1$ 
are Corson, then~$X$ is Corson. 

Suppose that the inductive hypothesis is true for all compact Hausdorff spaces of weight less than $\lambda$. 
Then $\lambda\geq \aleph_2$. 
Fix a compact Hausdorff space~$X$ such that all continuous images of $X$ of weight less than $\lambda$ are Corson. 
Since every Corson space is Fr\'echet and since $2^{\aleph_0}=\aleph_1<\lambda$, Proposition~\ref{P.Angelic} implies that $X$ is Fr\'echet. 

Fix a regular $\theta$  such that $X$ and $C(X)$ belong to $H_\theta$. 
By Lemma~\ref{L.split}, the set 
\[
\sfK=\{M\in \cPoH: X\in M, M\text{ splits } X\}
\]
includes a club. By Kueker's theorem (see e.g., \cite[Theorem~3.4]{foreman2010ideals}), 
there exists $f\colon H_\theta^{<\omega}\to H_\theta$ such that every $M\in \cPoH$ closed under $f$ belongs to~$\sfK$.
By Lemma~\ref{L.K}, if $M\prec H_\theta$, $X\in M$, and $M$ is closed under $f$, then $M$ splits $X$. 
Hence if in addition $|M|<\lambda$, then  $\McX$ 
is a Corson compact subspace of $X$. 

Let $M_\alpha$, for $\alpha<\lambda$, be an increasing chain of elementary submodels of~$H_\theta$ closed under $f$ such that $X\in M_0$ and $|M_\alpha|<\lambda$ for all $\alpha$. 
By the previous paragraph, these models satisfy the assumptions of Lemma~\ref{L.Recursive}. Since in addition $X$ is Fr\'echet, 
Lemma~\ref{L.Recursive} implies that $X$ is Corson. 
%
%
%
%
%
%
%
%
\end{proof}

\section{Some remarks on splitting models} 
\label{S.Splits} 

In the concluding section we collect a few observations on the notion of splitting models not  
required in the proofs of our main results. 
In all of the following examples, $\theta$ is a cardinal large enough regular to have $X\in H_\theta$.

\begin{example} \label{Ex.betaN} 
	There exists a compact Hausdorff space $X$ such that no countable $M\prec H_\theta$
	splits $X$. Take for example $X=\beta\bbN\setminus \bbN$, the \v Cech--Stone remainder of $\bbN$. 
	We only need to know that if $Z$ is a countable discrete subset of $X$, then the closure  of $Z$ is homeomorphic to $\beta\bbN$
and therefore of cardinality $2^{2^{\aleph_0}}$.

Suppose $M\prec H_\theta$ is countable (and certainly $X=\beta\bbN\setminus \bbN\in M$ if $\theta$ is large enough). 
A counting argument shows that \eqref{4.L.bad} of Lemma~\ref{L.bad} fails.  
First, $|\McX|=2^{2^{\aleph_0}}$. Second, $C_M$ is separable and therefore $|\cP(C_M)|\leq 2^{\aleph_0}$. 
\end{example}

\begin{example} \label{Ex.duplicate} 
	There is a compact, Hausdorff,  Fr\'echet (even first countable),  and separable  space  $X$ such that no countable $M\prec H_\theta$ splits $X$. 
	Let $X$ be $[0,1]\times \{0,1\}$ with the lexicographical ordering and the order topology. 
	It is easy to check that this space is compact, Hausdorff, first countable, and separable. Suppose $M\prec H_\theta $ is countable. 
	Then $\McX=X$ since a countable dense set of $X$ belongs to (and is therefore a subset of) $M$. 
	Choose $x\in [0,1]\setminus M$. Then $(x,0)$ and $(x,1)$ are not separated by the elements of~$C_M$, 
	and therefore $M$ fails \eqref{2.L.bad} of Lemma~\ref{L.bad}. 
\end{example}

Our last example is more specific and most relevant to 	Theorem~\ref{T.1}.

\begin{example} \label{Ex.MP} 
There exists a compact Hausdorff space $X$ and an increasing sequence of elementary submodels $M_n\prec H_\theta$ such that each $M_n$ splits $X$ but $\bigcup_n M_n$ does not split $X$. 
This space is based on \cite{magidor2017properties}, and we include the relevant details for reader's convenience. 

	Suppose $\lambda$ is an ordinal, $S\subseteq \lambda$ and every $\alpha\in S$ is a limit ordinal of cofinality $\omega$. 
For each $\alpha\in S$ fix a strictly  increasing sequence $p_n(\alpha)$, for $n\in \bbN$, of ordinals  such that $\sup_n p_n(\alpha)=\alpha$. 
Let $A_\alpha=\{p_n(\alpha): n\in \bbN\}$ and 
let $\fA(S)$ be the Boolean algebra of subsets of $\sfZ=\bigcup_{\alpha\in S} A_\alpha$ generated by all finite subsets of $\sfZ$ and $\{A_\alpha: \alpha\in S\}$. 

Let $X(S)$ be the Stone space of $\fA(S)$, i.e., the space of ultrafilters of~$\fA$. It is a compact Hausdorff space. 
Each ordinal $\xi<\lambda$ corresponds to a principal ultrafilter which is an isolated point $x(\xi)$ of $X$.
Each $\alpha\in S$ corresponds to a unique nonprincipal ultrafilter $y(\alpha)$ that concentrates on~$A_\alpha$. It satisfies
 $y(\alpha)=\lim_n x(p_n(\alpha))$. 
  Finally, a unique ultrafilter $z$ in $X(S)$ does not concentrate on any countable set. 
 
If $S$ is stationary then $X$ is not Corson (this was proved for $\lambda=\omega_2$ in \cite[Lemma~3.3]{magidor2017properties}, but the proof of the general case is identical). 
On the other hand, if $|\lambda|\leq \aleph_1$ and $S\cap \gamma$ is nonstationary for all $\gamma<\lambda$, then $\fA(S)$ is uniform Eberlein compact 
(\cite[Lemma~3.5]{magidor2017properties}). 

Therefore the existence of a nonreflecting stationary set\footnote{Here $S^2_0$ denotes the set of all ordinals below $\omega_2$ of cofinality $\omega$.}  $S\subseteq S^2_0$ implies that there exists a space $X=X(S)$ all of whose continuous images of weight not greater than $\aleph_1$ are uniform 
Eberlein compacta, but $X$ is not Corson. 

Since $X(S)$ is the Stone space of the Boolean algebra $\fA(S)$, the relevant \cstar-algebra $C(X)$ has a dense subset consisting of continuous functions with finite range (i.e., `step functions')
and $C(X)$ therefore does not provide any more information than $\fA(S)$. We will therefore work with $\fA(S)$ in the following.

Now suppose $M\prec H_\theta$ is such that $\lambda$ and $S$ belong to $M$. 
Then 
\[
M\cap X(S)=\{x(\xi): \xi\in M\}\cup \{y(\alpha): \alpha\in M\}\cup \{z\}
\]
and (note that $\alpha\in M$ implies $A_\alpha\subseteq M$) 
\[
\McXS=\{x(\xi): \xi\in M\}\cup \{y(\alpha): A_\alpha\cap M\text{ is infinite}\}\cup \{z\}. 
\]
Therefore if  $M\cap \lambda$ is an ordinal and it belongs to $S$, then $y(M\cap \lambda)$ belongs to $\McXS$ but not to $M\cap X(S)$. 
In particular, every $M\prec H_\theta$ such that $M\prec \lambda$ is an ordinal of uncountable cofinality splits $X(S)$. 

Now suppose that $S\subseteq S^2_0$ is stationary. 
We can then choose an increasing sequence $M_n\prec H_\theta$ for $n\in \bbN$ such that $M_n\cap \omega_2$ is an 
ordinal of cofinality~$\omega_1$, 
but $M=\bigcup_n M_n$ satisfies $M\cap \omega_2\in S$. 
Then each $M_n$ splits $X(S)$, but $M$ does not. 
\end{example} 

\section{Concluding Remarks} 
\label{S.Concluding} 

A compact Hausdorff space~$X$ is an \emph{Eberlein compactum} (or shortly, Eberlein)
if it is homeomorphic to a subspace of some Tychonoff cube $[0,1]^\kappa$
which has the property that for every $\xi<\kappa$ 
and every $\e>0$ the set $\{x\in X: x(\xi)>\e\}$ is finite. 
Every Eberlein compactum is clearly Corson. 
We do not know whether analog of Theorem~\ref{T.1} holds for Eberlein compacta. 
It is not difficult to prove that it holds for strong Eberlein compacta, and this can 
be easily extracted from \cite[Corollary~3.3]{kunen2003compact}. 

Our original proof of (a special case of) Theorem~\ref{T.1} used a very strong reflection principle obtained 
by adapting the proof of \cite[Theorem~1, (a) $\Rightarrow$ (c)]{magidor1982reflecting}.

\bibliographystyle{amsplain}
\bibliography{corson}

\end{document}